\documentclass[12pt]{article}
\usepackage{amsmath}
\usepackage{amsthm}
\usepackage{amssymb}
\usepackage{mathtools}
\usepackage{graphics}
\usepackage{enumerate} 
\usepackage{setspace}
 \newcommand{\be}{\begin{equation}}
\newcommand{\ee}{\end{equation}}
\newcommand{\ben}{\begin{eqnarray*}}
\newcommand{\een}{\end{eqnarray*}}
\newtheorem{examp}{\sc Example}
\newtheorem{remk}{\sc Remark}
\newtheorem{corol}{\sc Corollary}
\newtheorem{lemma}{\sc Lemma}
\newtheorem{theorem}{\sc Theorem}
\newtheorem{defn}{\sc Definition}
\newcommand{\bt}{\begin{theorem}}
\newcommand{\et}{\end{theorem}}
\newcommand{\bl}{\begin{lemma}}
\newcommand{\el}{\end{lemma}}
\newcommand{\bed}{\begin{defn}}
\newcommand{\eed}{\end{defn}}
\newcommand{\brem}{\begin{remk}}
\newcommand{\erem}{\end{remk}}
\newcommand{\bex}{\begin{examp}}
\newcommand{\eex}{\end{examp}}
\newcommand{\bcl}{\begin{corol}}
\newcommand{\ecl}{\end{corol}}

\newcommand{\NI}{\noindent}

\newcommand{\al}{\alpha}

\newcommand{\dsp}{\displaystyle}
\newcommand{\vsp}{\vskip 0.5em}

\theoremstyle{definition}
\theoremstyle{remark}

\numberwithin{equation}{section}
\numberwithin{theorem}{section}
\numberwithin{lemma}{section}

\begin{document}
\title {On Some Properties of $K$- type Block Matrices in the context of Complementarity Problem } 
\author{A. Dutta$^{a, 1}$,  A. K. Das$^{b}$\\
\emph{\small $^{a}$Jadavpur University, Kolkata , 700 032, India.}\\	
\emph{\small $^{b}$Indian Statistical Institute, 203 B. T.
	Road, Kolkata, 700 108, India.}\\
\emph{\small $^{1}$Email: aritradutta001@gmail.com}}
\date{}
\maketitle

\date{}
\maketitle
	\begin{abstract}
	\NI In this article we introduce $K$-type block matrices which include two new classes of block matrices namely block triangular $K$-matrices and hidden block triangular $K$-matrices. We show that the solution of linear complementarity problem with $K$-type block matrices can be obtained by solving a linear programming problem. We show that block triangular $K$-matrices satisfy least element property.  We prove that hidden block triangular $K$-matrices are $Q_0$ and processable by Lemke's algorithm. The purpose of this article is to study  properties of  $K$-type block matrices in  the context of the solution of linear complementarity problem.
	\noindent \\
	\NI{\bf Keywords:} $Z$-matrix, Hidden $Z$-matrix, linear programming problem, linear complementary problem, semi-sublattice,  $P$-matrix, $Q_0$-matrix.\\ 
	
	\NI{\bf AMS subject classifications:} 90C33, 90C51, 15A39, 15B99.

\end{abstract}

\section{Introduction}
The linear complementarity problem is a combination of linear and nonlinear system of inequalities and equations. The problem may be stated as follows:
Given $M\in R^{n\times n}$ and  a vector  $\,q\,\in\,R^{n},\,$ the  linear complementarity problem, LCP$(M,q)$  is the problem of finding a solution $w\;\in R^{n}\;$ and  $z\;\in R^{n}\;$ to the following system of linear equations and inequalities:
\be \label{lcp1}
\dsp {w\,-\,M z\;=\;q,\;\;w\geq 0,\;z\geq 0}
\ee
\be \label{lcp2}
\dsp {w^{T}\,z\,=\;0}.
\ee
 In complementarity theory several matrix classes are considered due to the study of theoritical properties, applications and its solution methods. For details see \cite{neogy2006some}, \cite{neogy2013weak}, \cite{neogy2005almost}, \cite{neogy2011singular}, \cite{das2017finiteness}. The complementarity problem plays an important role in the  formulation of structured stochastic game problems. For details see \cite{mondal2016discounted}, \cite{neogy2008mathematical}, \cite{neogy2008mixture}, \cite{neogy2016optimization}, \cite{neogy2005linear}. The complementarity problem establishes an important connections with multiobjective programming problem for KKT point and the solution point \cite{mohan2004note}. The complementarity problems are considered with respect to principal pivot transforms and pivotal method to its solution point of view. For details see \cite{neogy2012generalized}, \cite{das2016properties}, \cite{neogy2005principal}.
It is well known that the linear complementarity problem can be solved by a linear program if $M$ or its inverse is a $Z$-matrix, i.e. a real square matrix with non-positive off diagonal elements. A number of authors have considered the special case of the linear complementarity problem under the restriction that $M$ is a $Z$-matrix. Chandrasekharan \cite{pch} considered $Z$-matrix solving a sequence of linear inequalities. Lemke's algorithm is a well known technique for solving linear complementarity problem \cite{cps}. Mangasarian \cite{mangasarian} showed that the following linear program
\be \label{equationz}
\begin{array}{ll}
	\text{minimize} & p^Tu\\
	\text{subject to} & q + Mu \geq 0,\\
	& u \geq 0 
\end{array}
\ee
for an easily determined $p \in R^n$ solves the linear complementarity problem for a number of special cases specially when $M$ is a $Z$-matrix. Mangasarian \cite{mangasarian} proved that least element of the polyhedral set $\{u: q + Mu \geq 0, u \geq 0\}$ in the sense of Cottle-Veinott  can be obtained by a single linear program. It is well known that the quadratic programming problem $$\begin{array}{ll}
\text{minimize} & q^Tu + \frac{1}{2} u^TMu\\
\text{subject to} & u \geq 0
\end{array}$$
can be formulated as a linear complementarity problem when $M$ is symmetric positive semidefinite. Mangasarian showed that this problem can be solved using single linear program if $M$ is a $Z$-matrix. Hidden $Z$-matrices are the extension of $Z$-matrices. A matrix $M$ is said to be a hidden $Z$-matrix if $\exists$ two $Z$-matrices $X$ and $Y$ such that
\begin{enumerate}[(i)]
		\item $MX = Y$
	\item $r^TX + s^TY > 0,$ for some $r, s \geq 0.$
\end{enumerate}
For details, see \cite{jana2019hidden}, \cite{jana2021more}.
In this paper we introduce block triangular $K$-matrix and hidden block triangular $K$-matrix. We call these two classes collectively as $K$-type block matrix. We discuss the class of $K$-type block matrices in solution aspects for linear complementarity problem.\\
 The paper is organized as follows. Section 2 presents some basic notations, definitions and results. In section 3, we establish some results of these two matrix classes. We show that a linear complementarity problem with block triangular $K$-matrix and hidden block triangular $K$-matrix can be solved using linear programming problem.
\section{Preliminaries}
\noindent We denote the $n$ dimensional real space by $R^n.$ $R^n_+$ denotes the nonnegative orthant of $R^n.$ We consider vectors and matrices with real entries. Any vector $x\in R^{n}$ is a column vector and  $x^{T}$ denotes the row transpose of $x.$ $e$ denotes the vector of all $1.$ A matrix is said to be nonnegative or $M \geq 0$ if $m_{ij} \geq 0 \ \forall \ i,j.$ A matrix is said to be positive if $m_{ij} > 0 \ \forall \ i,j.$ Let $M$ and $N$ be two matrices with $M \geq N,$ then $M-N \geq 0.$  If $M$ is a matrix of order $n,$ $\al \subseteq \{1, 2, \cdots, n\}$ and $\bar{\al} \subseteq \{1, 2, \cdots, n\} \setminus \al$ then $M_{\al \bar{\al}}$ denotes the submatrix of $M$ consisting of only the rows and columns of $M$ whose indices are in $\al$ and $\bar{\al}$ respectively. $M_{\al \al}$ is called a principal submatrix of M and det$(M_{\al \al})$ is called a principal minor of $M.$ Given a matrix $M \in R^{n \times n}$ and a vector $q \in R^n,$ we define the feasible set FEA$(M, q)$ $= \{z \in R^n : z \geq 0, q + Mz \geq 0\}$ and the solution set of LCP$(M, q)$ by SOL$(M, q)$ $=\{z \in \text{FEA}(M, q) : z^T(q + Mz) = 0\}.$

We state the results of two person matrix games in linear system with complementary conditions due to von Neumann \cite{von} and Kaplansky \cite{kaplansky}. The results say that there exist $\bar{x} \in R^m, \bar{y} \in R^n$ and $v \in R$ such that 
\begin{center}
	$\sum_{i=1}^{m} \bar{x}_i a_{ij} \leq v, \; \forall\; j = 1, 2, \cdots, n,$\\
	$\sum_{j=1}^{n} \bar{y}_j a_{ij} \geq v, \; \forall\; i = 1, 2, \cdots, m.$
\end{center}
The strategies $(\bar{x}, \bar{y})$ are said to be optimal strategies for player I and player II and $v$ is said to be minimax value of game.
We write $v(A)$ to denote the value of the game corresponding to the payoff matrix $A.$ The value of the game, $v(A)$ is positive(nonnegative) if there exists a $0 \neq x \geq 0$ such that $Ax > 0\;(Ax \geq 0).$ Similarly, $v(A)$ is negative(nonpositive) if there exists a $0 \neq y \geq 0$ such that $y^TA < 0\;(y^TA \leq 0).$ 

A matrix $M \in R^{n \times n}$ is said to be\\
$-$ $PSD$-matrix if $x^TMx \geq 0 \ \forall \ 0 \neq x \in R^n.$\\
$-$ $P\,(P_{0})$-matrix if all its principal minors are positive (nonnegative). \\
 $-$ $S$-matrix \cite{pang} if there exists a vector $x>0$ such that $Mx>0$ and $\bar{S}$-matrix if all its principal submatrices are $S$-matrix.\\
$-$ $Z$-matrix if off-diagonal elements are all non-positive and $K\,(K_0)$-matrix if it is a $Z$-matrix as well as $P\,(P_0)$-matrix.\\
$-$ $Q$-matrix if for every $q,$ LCP$(M, q)$ has at least one solution. \\
$-$ $Q_0$-matrix if for FEA$(q, A) \neq \emptyset$$\implies$SOL$(q, A) \neq \emptyset.$ 
\vsp
Now we give some definitions, lemmas, theorems which will be required for discussion in the next section. 
\begin{lemma}\cite{cps}\label{pptt}
	If $A$ is a $P$-matrix, then $A^T$ is also $P$-matrix.
\end{lemma}
\begin{lemma}
	Let $A$ be a $P$-matrix.  Then $v(A)>0.$
\end{lemma}
\begin{defn}\cite{cps}
	A subset $S$ of $R^n$ is called a meet semi-sublattice(under the componentwise ordering of $R^n$) if for any two vectors $x$ and $y$ in $S$, their meet, the vector $z=\min(x,y)$  belongs to $S$.
	\end{defn}
\begin{defn}\cite{fiedler}
	The spectral radius $\sigma(M) $ of $M$ is defined as the maximum of the  moduli $|\lambda|$ of all proper values $\lambda$ of $M$.
\end{defn}
 
\begin{lemma}\cite{fiedler}
	Let $M$ be a nonnegative matrix. Then there exists a proper value $p(M)$ of $M$, the Perron root of $M$, such that $p(M)\geq 0$ and  $|\lambda|\leq p(M)$ for every proper value $\lambda$ of $M.$ If \ $0 \leq M \leq N$ then \ $p(M) \leq p(N).$ Moreover, if $M$ is irreducible, the Perron-Frobenius root $p(M)$ is positive, simple and the corresponding proper value may be chosen positive. According to the  Perron-Frobenius theorem, we have $\sigma(M)=p(M)$ for nonnegative matrices.
\end{lemma}
\begin{defn}
	A matrix $W$ is said to have dominant principal diagonal if $|w_{ii}| >\sum_{k\neq i}|w_{ik}|$ for each $i.$
\end{defn}	    

\begin{lemma}\cite{fiedler}
	If $W$ is a matrix with dominant principal diagonal, then $\sigma(I-H^{-1}W)<1,$ where $H$ is the diagonal of $W$.
\end{lemma}
\begin{theorem}\cite{fiedler}\label{pvp}
	The following four properties of a matrix are equivalent:\\
	(i) All principal minors of $M$ are positive.\\
	(ii) To every vector $x \neq 0$ there exists an index $k$ such that $x_ky_k >0$ where $y=Mx$.\\
	(iii)To every vector $x \neq 0$ there exists a diagonal matrix $D_x$ with positive diagonal elements such that the inner product $(Mx,D_xx) >0.$\\
	(iv)To every vector $x \neq 0$ there exists a diagonal matrix $H_x \geq 0$ such that the inner product $(Mx,H_xx) >0.$\\
	(v)Every real proper value of $M$ as well as of each principal minor of $M$ is positive.	
\end{theorem}
\begin{lemma}\cite{cps}\label{sub}
	If $F$ is a nonempty meet semi-sublattice that is closed and bounded below, then $F$ has a least element.
\end{lemma}

\begin{lemma}\cite{mangasarianlinear}\label{lemm}
	If $z$ solves the linear program $\min p^Tz $ subject to $Mz+q \geq 0, z \geq 0$ and if the corresponding optimal dual variable $y$ satisfies $(I-M^T)y+p>0,$ then $z$ solves the linear complementarity problem LCP$(M,q).$
\end{lemma}
\section{Main Results}

In this paper we introduce block triangular $K$-matrix and hidden block triangular $K$-matrix, which are defined as follows: A matrix $M \in R^{mn \times mn}$ is said to be a block triangular $K$-matrix if it is formed with block of $K$-matrices $M_{ij}\in R^{m \times m},$ either in upper triangular forms or in lower triangular forms. Here $i$ and $j$ varry from $1$ to $n.$ For block upper triangular form of $M,$ the blocks $M_{ij}= 0$ for $i<j$ and for block lower triangular form of $M,$ the blocks $M_{ij}= 0$ for $i>j.$ \\
 \begin{spacing}{1.5}
Consider $M=\left[\begin{array}{c|c|c} 
\begin{array}{rr}
1 & -1\\
-1.5 & 2\\
\end{array} & \begin{array}{rr}
0 & 0\\
0 & 0\\
\end{array} & \begin{array}{rr}
0 & 0\\
0 & 0\\
\end{array}\\
\hline
\begin{array}{rr}
3 & -1\\
-1 & 4\\
\end{array} & \begin{array}{rr}
1 & -1\\
-0.75 & 1\\
\end{array} & \begin{array}{rr}
0 & 0\\
0 & 0\\
\end{array}\\
\hline
\begin{array}{rr}
1 & -1\\
-0.5 & 1\\
\end{array} & \begin{array}{rr}
1 & -0.5\\
-0.5 & 1\\
\end{array} & \begin{array}{rr}
5 & -1\\
-10 & 6\\
\end{array}\\
\end{array}\right], $ \end{spacing} \NI which is a block triangular $K$-matrix.\\
 The matrix $N\in R^{mn\times mn}$ is said to be  hidden block triangular $K$-matrix if there exist two block triangular $K$-matrices $X$ and $Y$ such that $NX=Y.$ $N$ is formed with block matrices either in upper triangular forms or in lower triangular forms. For block upper triangular form of $N,$ the blocks $N_{ij}= 0$ for $i<j$ and $X, Y$ are formed with $K$ matrices in upper triangular form. Similarly for block lower triangular form of $N,$ the blocks $N_{ij}= 0$ for $i>j$ and $X, Y$ are formed with $K$ matrices in lower triangular form.\\
\begin{spacing}{1.5}
Consider $N=$$\left[\begin{array}{c|c} 
\begin{array}{rr}
-1 & -1\\
5 & 4\\
\end{array} & \begin{array}{rr}
0 & 0\\
0 & 0\\
\end{array} \\
\hline
\begin{array}{rr}
-4.5 & -3\\
4 & 3.875\\
\end{array} & \begin{array}{rr}
1 & 0.5\\
-0.25 & 0.3125\\
\end{array} \\
\end{array}\right], $ \\
\vsp \NI
 $X=\left[\begin{array}{c|c} 
\begin{array}{rr}
2 & -1\\
-3 & 2\\
\end{array} & \begin{array}{rr}
0 & 0\\
0 & 0\\
\end{array} \\
\hline
\begin{array}{rr}
3 & 0\\
-2 &1\\
\end{array} & \begin{array}{rr}
4 & -1\\
0 & 4\\
\end{array} \\
\end{array}\right]$
and $Y=\left[\begin{array}{c|c} 
\begin{array}{rr}
1 & -1\\
-2 & 3\\
\end{array} & \begin{array}{rr}
0 & 0\\
0 & 0\\
\end{array} \\
\hline
\begin{array}{rr}
2 & -1\\
0 & 1\\
\end{array} & \begin{array}{rr}
4 & 0\\
-1 & 1\\
\end{array} \\	
\end{array}\right],$  \end{spacing} \NI such that $NX=Y.$ Then $N$ is a hidden block triangular $K$-matrix.
\begin{theorem}\label{p}
	Let $M$ be a block triangular $K$-matrix. Then LCP$(M,q)$ is processable by Lemke's algorithm.
\end{theorem}
\begin{proof}
	Let $M$ be  a block triangular $K$-matrix. Then $\exists \  z \in R^n$ such that $z_i(Mz)_i\leq 0 \ \forall i \implies (z_1)_i(M_{11}z_1)_i \leq 0 \ \forall i \implies z_1=0,$ as $M_{11}\in K;$ $(z_2)_i(M_{21}z_1+M_{22}z_2)_i \leq 0 \ \forall i \implies (z_2)_i(M_{22}z_2)_i \leq 0 \ \forall i \implies z_2=0,$ as $M_{22}\in K.$ In similar way $(z_n)_i(M_{n1}z_1+M_{n2}z_2+\cdots+M_{nn}z_n)_i \leq 0 \ \forall i \implies (z_n)_i(M_{nn}z_n)_i \leq 0 \ \forall i \implies z_n=0,$ as $M_{nn}\in K$ and $z_1=z_2=\cdots=z_{n-1}=0.$ Hence $z=0.$ So $M$ is a $P$-matrix. Therefore LCP$(M,q)$ is processable by Lemke's algorithm.
\end{proof}
\begin{remk}\cite{jana}
		Let $M$ be a block triangular $K$-matrix. Then LCP$(M,q)$ is solvable by criss-cross method.
\end{remk}
\begin{theorem}\label{lattice}
	If $M$ is a block triangular $K$-matrix and $q$ is an arbitrary vector, then the feasible region of LCP$(M,q)$ is a meet semi-sublattice.
\end{theorem}
\begin{proof}
	\begin{spacing}{1.5}
	Let $F=$FEA$(M,q).$ Let $x=\left[\begin{array}{rrrrr}
	x_1\\
	x_2\\
	x_3\\
	\vdots\\
	x_n\\
	\end{array}\right],y=\left[\begin{array}{rrrrr}
	y_1\\
	y_2\\
	y_3\\
	\vdots\\
	y_n\\
	\end{array}\right] \in F$ \end{spacing} \NI are two feasible vectors. So $x \geq 0,y\geq 0,Mx+q\geq 0,My+q \geq 0.$ \\ 
	\vsp \NI
	\begin{spacing}{1.5}
\NI	Let $z=\left[\begin{array}{rrrrr}
	z_1\\
	z_2\\
	z_3\\
	\vdots\\
	z_n\\
	\end{array}\right]=\text{min}(x,y).$ Then \\
	\vsp \NI
	$ Mx+q=\left[\begin{array}{rrrrr}
	M_{11}x_1+q_1\\
	M_{21}x_1+M_{22}x_2+q_2\\
	M_{31}x_1+M_{32}x_2+M_{33}x_3+q_3\\
	\vdots\\
	M_{n1}x_1+M_{n2}x_2+M_{n3}x_3+ \cdots +M_{nn}x_n+q_n\\
	\end{array}\right]\geq 0.\\ $ \end{spacing} $\implies x_1\in \text{FEA}(M_{11},q_1), x_2 \in \text{FEA}(M_{22},M_{21}x_1+q_1), \cdots, x_n \in \text{FEA}(M_{nn},M_{n1}x_1+M_{n2}x_2+\cdots+M_{n(n-1)}x_{n-1}+q_n).$  In similar way $My+q \geq 0 \implies  y_1\in \text{FEA}(M_{11},q_1), \\ y_2 \in \text{FEA}(M_{22},M_{21}x_1+q_1),\cdots, y_n \in \text{FEA}(M_{nn},M_{n1}x_1+M_{n2}x_2+\cdots+M_{n(n-1)}x_{n-1}+q_n).$ Suppose $z=\text{min}(x,y) \implies z_1=\text{min}(x_1,y_1), z_2=\text{min}(x_2,y_2),\cdots, z_n=\text{min}(x_n,y_n).$ $M_{ij} \in K \implies z_1 \in \text{FEA}(M_{11},q_1) \implies M_{11}z_1+q_1 \geq 0, z_2 \in \text{FEA}(M_{22},M_{21}z_1+q_2) \implies M_{22}z_2+M_{21}z_1+q_2 \geq 0,\cdots,  z_n \in \text{FEA}(M_{nn},M_{n1}z_1+M_{n2}z_2+\cdots+M_{n(n-1)}z_{n-1}+q_n) \implies M_{n1}z_1+M_{n2}z_2+\cdots+M_{n(n-1)}z_{n-1}+M_{nn}z_n+q_n\geq 0.$ So $z=\text{min}(x,y)\in \text{FEA}(M,q).$ Hence the feasible region is a meet semi-sublattice.
	\end{proof}  
	Cottle et al.\cite{cps} showed that if $F$ is a nonempty meet semi-sublattice, that is closed and bounded below, then $F$ has a least element by lemma \ref{sub}. Now we show that if the LCP$(M,q)$ is feasible, where $M$ is a block triangular $K$-matrix, then FEA$(M,q)$ contains a least element $u.$

	\begin{theorem}
		Let $M$ be a block triangular $K$-matrix and $q$ be an arbitrary vector. If the LCP$(M,q)$ is feasible, then FEA$(M,q)$ contains a least element $u,$ where $u$ solves the LCP$(M,q).$
	\end{theorem}
\begin{proof}
	Let $F=$FEA$(M,q).$ By theorem \ \ref{lattice}, $F$ is a meet semi-sublattice. Let  LCP$(M,q)$ be feasible. Then the set $F$ is obviously nonempty and bounded below by zero. Then the existence of the least element $l=\left[\begin{array}{rrrrr}
	l_1\\
	l_2\\
	l_3\\
	\vdots\\
	l_n\\
	\end{array}\right]$ follows from lemma \ref{sub}. That is $l =\left[\begin{array}{rrrrr}
	l_1\\
	l_2\\
	l_3\\
	\vdots\\
	l_n\\
	\end{array}\right] \leq \left[\begin{array}{rrrrr}
	x_1\\
	x_2\\
	x_3\\
	\vdots\\
	x_n\\
	\end{array}\right] =x \ \forall \ x \in F$ and $l \in F.$\\ 
	\vsp
	 \NI Let $F_i = \ $FEA$(M_{ii},M_{i(i-1)}z_{i-1}+M_{i(i-2)}z_{i-2}+\cdots+M_{i2}z_2+M_{i1}z_1+q_i).$ Now it is clear that $y_1 \in F_1, y_2 \in F_2, \cdots, y_n \in F_n,$ where $y=\left[\begin{array}{rrrrr}
	y_1\\
	y_2\\
	y_3\\
	\vdots\\
	y_n\\
	\end{array}\right] \in F.$ As $M_{ii} $ are $Z$- matrices, $l_i$ is the least element of $F_i \ \forall \ i \in \{1,2,\cdots,n\}$ and $l_i$ solves LCP$(M_{ii},M_{i(i-1)}z_{i-1}+M_{i(i-2)}z_{i-2}+\cdots+M_{i2}z_2+M_{i1}z_1+q_i).$ So $l=\left[\begin{array}{rrrrr}
	l_1\\
	l_2\\
	l_3\\
	\vdots\\
	l_n\\
	\end{array}\right]$ solves LCP$(M,q).$
\end{proof}

\vsp
	Mangasarian \cite{mangasarianlinear} showed that if $z$ solves the linear program, \  $\min p^Tz $ subject to $Mz+q \geq 0, z \geq 0$ and if the corresponding optimal dual variable $y$ satisfies $(I-M^T)y+p>0,$ then $z$ solves the linear complementarity problem LCP$(M,q)$ by lemma \ref{lemm}. Here we show that if LCP$(M,q)$ with $M,$ a block triangular $K$-matrix, has a solution which can be obtained by solving the linear program $\min $ \ $p^Tx $ subject to $Mx+q \geq 0,$ $x \geq 0.$

\begin{theorem}
	The linear complementarity problem LCP$(M,q),$ where $M$ is a block triangular $K$-matrix, has a solution which can be obtained by solving the linear program $\min $ \ $p^Tx $ subject to $Mx+q \geq 0,$ $x \geq 0,$ where $p=r \geq 0$ and $Z_1$ is a block triangular $K$-matrix with $r^TZ_1>0.$
\end{theorem}
\begin{proof}
	Let $M$ be a block triangular $K$-matrix. The linear program,  $\min $ \ $p^Tx $ subject to $Mx+q \geq 0,$ $x \geq 0$ and the dual linear program, $\max -q^Ty$ subject to $-M^Ty+p \geq 0,\  y \geq 0$ have solutions $x$ and $y$ respectively. $M$ can be written as $D-U,$ where 	\begin{spacing}{1.5} \NI $D=\left[\begin{array}{rrrrr} 
	D_{11} & 0 & 0 & \cdots & 0\\
	D_{21} & D_{22} & 0  & \cdots & 0\\	
	D_{31} & D_{32} & D_{33} & \cdots & 0\\
	\vdots & \vdots & \vdots & \vdots & \vdots\\
	D_{n1} & D_{n2} & D_{n3} & \cdots & D_{nn}\\
	\end{array}\right],$\end{spacing} \NI $ D_{ij}$'s are diagonal matrices with positive entries and  
		\begin{spacing}{1.5} \NI
	$U=\left[\begin{array}{rrrrr} 
	U_{11} & 0 & 0 & \cdots & 0\\
	U_{21} & U_{22} & 0  & \cdots & 0\\	
	U_{31} & U_{32} & U_{33} & \cdots & 0\\
	\vdots & \vdots & \vdots & \vdots & \vdots\\
	U_{n1} & U_{n2} & U_{n3} & \cdots & U_{nn}\\
	\end{array}\right],$ \end{spacing} \NI $ U_{ij}$'s are matrices with nonnegative entries.
Consider $Z_1=D-V,$ a block triangular $K$-matrix and the matrix product $MZ_1=D-W$, where 
	\begin{spacing}{1.5} \NI
 $V=\left[\begin{array}{rrrrr} 
V_{11} & 0 & 0 & \cdots & 0\\
V_{21} & V_{22} & 0  & \cdots & 0\\	
V_{31} & V_{32} & V_{33} & \cdots & 0\\
\vdots & \vdots & \vdots & \vdots & \vdots\\
V_{n1} & V_{n2} & V_{n3} & \cdots & V_{nn}\\
\end{array}\right],$ \end{spacing} \NI $ V_{ij}$'s are matrices with nonnegative entries and 	\begin{spacing}{1.5} \NI $W=\left[\begin{array}{rrrrr} 
W_{11} & 0 & 0 & \cdots & 0\\
W_{21} & W_{22} & 0  & \cdots & 0\\	
W_{31} & W_{32} & W_{33} & \cdots & 0\\
\vdots & \vdots & \vdots & \vdots & \vdots\\
W_{n1} & W_{n2} & W_{n3} & \cdots & W_{nn}\\
\end{array}\right],$ \end{spacing} \NI $ W_{ij}$'s are matrices with nonnegtive entries. Since $Z_1$ is a block triangular  $K$-matrix, it is a $P$-matrix. Hence $v(Z_1)>0$ and by lemma \ref{pptt} $v({Z_1}^T)>0$. Let $r \geq 0$ be the value of $Z_1^T,$ then $r^TZ_1>0.$ Now $0<r^TZ_1=p^TZ_1= p^TZ_1+y^T(-MZ_1+D-W)=(p^T-y^TM)Z_1+y^T(D-W)=(p^T-y^TM)(D-V)+y^T(D-W)\leq (p^T-y^TM+y^T)D$ as $p^T-y^TM \geq 0, y \geq 0, U\geq 0,V \geq 0.$ This implies $(I-M^T)y+p>0,$ since $D_{ij}$'s are positive diagonal matrices. So by lemma \ref{lemm} $x$ solves LCP$(M,q),$ which is a solution of $\min $ \ $p^tx $ subject to $Mx+q \geq 0,$ $x \geq 0.$ 
\end{proof}
\begin{corol}
	The solution of linear complementarity problem LCP$(M,q),$ with $M \in$ block triangular $K$-matrix can be obtained by solving the linear program $\min $ \ $e^Tx $ subject to $Mx+q \geq 0,$ $x \geq 0.$ 
\end{corol}
\begin{proof}
	The identity matrix $I$  itself is a block triangular $K$-matrix. Then $e^TI>0.$
\end{proof}
\begin{theorem}
	Let $M$ be a block triangular $K$-matrix. Then $M^{-1}$ exists and $M^{-1}\geq 0.$
\end{theorem}
\begin{proof}
Assume that $Q=I-tM \geq 0, t>0.$ Let $p(Q)$ be the Perron-root of $Q$. Then we have $\det [(1-p(Q))I - tM]=\det [Q-p(Q)I]=0.$ By theorem \ref{pvp}, $0<p(Q)<1.$ Thus the series $I+Q+Q^2+\cdots$ converges to the matrix $(I-Q)^{-1}=(tM)^{-1}\geq0,$ since $Q^k \geq 0$ for $k=1,2,\cdots.$ Therefore $M^{-1}$ exists and $M^{-1}\geq0.$
\end{proof}
\begin{theorem}
	Let $N$ be a block triangular $K$-matrix and $M$ ba a $Z$-matrix such that $M\leq N.$ Then both $M^{-1} $ and $N^{-1} $ exist and $M^{-1}\geq N^{-1} \geq 0.$\\

\end{theorem}
\begin{proof}
	Let $N$ be a block triangular $K$-matrix and $M$ ba a $Z$-matrix such that $M\leq N.$
Assume that $R=I-\alpha N \geq 0, \alpha>0.$ Then  $S=I-\alpha M \geq R \geq 0.$ Let $p(R)$ be a Perron root of $R$. Then we have $\det[(1-p(R))I - \alpha N]=\det [R-p(R)I]=0.$ By theorem \ref{pvp}, $0<p(R)<1.$ Thus the series $I+R+R^2+\cdots$ converges to the matrix $(I-R)^{-1}=(\alpha N)^{-1}.$ Since $S^k \geq R^k \geq 0,$ for $k=1,2,\cdots,$   the series $I+S+S^2+\cdots$ converges to the matrix $(I-S)^{-1}=(\alpha M)^{-1}.$ Therefore $M^{-1}$ and $N^{-1}$ exist and $M^{-1}\geq N^{-1} \geq 0.$\\
\end{proof}
\begin{corol}
	Assume that $M,N$ are block triangular $K$-matrices such that $M\leq N.$ Then both $M^{-1} $ and $N^{-1} $ exist and $M^{-1}\geq N^{-1} \geq 0.$\\
\end{corol}

\begin{theorem}
	Let $N$ be a hidden block triangular $K$- matrix. Then every diagonal block of $N$ is a hidden $Z$- matrix.
\end{theorem}
\begin{proof}
	Let $N$ be a hidden block triangular $K$- matrix with $NX=Y,$ where $X$ and $Y$ are block triangular $K$-matrices. Let \begin{spacing}{1.5} \NI  $N=\left[\begin{array}{rrrrr} 
	N_{11} & 0 & 0 & \cdots & 0\\
	N_{21} & N_{22} & 0  & \cdots & 0\\	
	N_{31} & N_{32} & N_{33} & \cdots & 0\\
	\vdots & \vdots & \vdots & \vdots & \vdots\\
	N_{n1} & N_{n2} & N_{n3} & \cdots & N_{nn}\\
	\end{array}\right],$ \end{spacing} \begin{spacing}{1.5} \NI $X=\left[\begin{array}{rrrrr} 
	X_{11} & 0 & 0 & \cdots & 0\\
	X_{21} & X_{22} & 0  & \cdots & 0\\	
	X_{31} & X_{32} & X_{33} & \cdots & 0\\
	\vdots & \vdots & \vdots & \vdots & \vdots\\
	X_{n1} & X_{n2} & X_{n3} & \cdots & X_{nn}\\
	\end{array}\right]$  and \end{spacing} \begin{spacing}{1.5} \NI $Y=\left[\begin{array}{rrrrr} 
	Y_{11} & 0 & 0 & \cdots & 0\\
	Y_{21} & Y_{22} & 0  & \cdots & 0\\	
	Y_{31} & Y_{32} & Y_{33} & \cdots & 0\\
	\vdots & \vdots & \vdots & \vdots & \vdots\\
	Y_{n1} & Y_{n2} & Y_{n3} & \cdots & Y_{nn}\\
	\end{array}\right].$ \end{spacing} \NI The block diagonal of $NX$ are $N_{ii}X_{ii}$ for $i \in \{1,2,\cdots n\}.$ So $N_{ii}X_{ii}=Y_{ii}$ for $i \in \{1,2,\cdots n\}.$ $X_{ii}, Y_{ii}$ are $K$-matrices. Then $X_{ii}^T, Y_{ii}^T$ are also $K$-matrices. So $v(X_{ii}^T)>0,$ $v(Y_{ii}^T)>0.$ Let $r_i,s_i \in {R^{m}}_+ $ such that $X_{ii}^Tr_i + Y_{ii}^Ts_i >0 \implies r_i^TX_{ii}+ s_i^TY_{ii}>0.$ Hence the block diagonals of $N$ are hidden $Z$-matrices.
\end{proof}
\begin{theorem}
	Let $N$ be a hidden block triangular $K$-matrix. Then every determinant of block diagonal matrices of $N$ are positive.
\end{theorem}
\begin{proof}
	Let $N$ be a hidden block triangular $K$-matrix with $NX=Y,$ where $X$ and $Y$ are block triangular $K$-matrices. Let  \begin{spacing}{1.5} \NI $N=\left[\begin{array}{rrrrr} 
	N_{11} & 0 & 0 & \cdots & 0\\
	N_{21} & N_{22} & 0  & \cdots & 0\\	
	N_{31} & N_{32} & N_{33} & \cdots & 0\\
	\vdots & \vdots & \vdots & \vdots & \vdots\\
	N_{n1} & N_{n2} & N_{n3} & \cdots & N_{nn}\\
	\end{array}\right],  $ \end{spacing}  
\begin{spacing}{1.5} \NI
$X=\left[\begin{array}{rrrrr} 
	X_{11} & 0 & 0 & \cdots & 0\\
	X_{21} & X_{22} & 0  & \cdots & 0\\	
	X_{31} & X_{32} & X_{33} & \cdots & 0\\
	\vdots & \vdots & \vdots & \vdots & \vdots\\
	X_{n1} & X_{n2} & X_{n3} & \cdots & X_{nn}\\
	\end{array}\right]$ and    $Y=\left[\begin{array}{rrrrr} 
	Y_{11} & 0 & 0 & \cdots & 0\\
	Y_{21} & Y_{22} & 0  & \cdots & 0\\	
	Y_{31} & Y_{32} & Y_{33} & \cdots & 0\\
	\vdots & \vdots & \vdots & \vdots & \vdots\\
	Y_{n1} & Y_{n2} & Y_{n3} & \cdots & Y_{nn}\\
	\end{array}\right].$ \end{spacing}
 \NI The block diagonal of $NX$ are $N_{ii}X_{ii}$ for $i \in \{1,2,\cdots n\}.$ So $N_{ii}X_{ii}=Y_{ii}$ for $i \in \{1,2,\cdots n\}.$ $X_{ii}, Y_{ii}$ are $K$-matrices. Then $\det (X_{ii}),\det (Y_{ii})>0 \ \forall \ i.$ Hence $\det (N_{ii})>0 \ \forall \ i.$  
\end{proof}
\begin{corol}
	Every block triangular $K$-matrix is a hidden block triangular $K$-matrix.
\end{corol}
\begin{proof}
	Let $M$ be a block triangular $K$-matrix. Taking $X=I,$ the identity matrix, it is clear that $M$ is a hidden block triangular $K$-matrices.
	\end{proof}
\begin{theorem}\label{tthh}
	The linear complementarity problem LCP$(N,q),$ where $N$ is a hidden block triangular $K$-matrix with $NX=Y, \ X \ \text{and} \ Y$ are block triangular $K$-matrices, has a solution which can be obtained by solving the linear program $\min $ \ $p^Tx $ subject to $Nx+q \geq 0,$ $x \geq 0,$ where $p=r+N^Ts \geq 0$ and $r,s \geq 0$ such that $X^Tr>0$ and $Y^Ts>0.$
\end{theorem}
\begin{proof}
	
		Let $N$ be a hidden block triangular $K$- matrix with $NX=Y,$ where $ \ X \ \text{and} \ Y$ are block triangular $K$-matrices. The linear program, $\min $ \ $p^Tx $ subject to $Nx+q \geq 0,$ $x \geq 0$ and the dual linear program, $\max -q^Ty$ subject to $-N^Ty+p \geq 0,\  y \geq 0$ have solutions $x$ and $y$ respectively. $X$ can be written as $D-U,$ where\\
		\begin{spacing}{1.5} \NI
		 $D=\left[\begin{array}{rrrrr} 
	D_{11} & 0 & 0 & \cdots & 0\\
	D_{21} & D_{22} & 0  & \cdots & 0\\	
	D_{31} & D_{32} & D_{33} & \cdots & 0\\
	\vdots & \vdots & \vdots & \vdots & \vdots\\
	D_{n1} & D_{n2} & D_{n3} & \cdots & D_{nn}\\
	\end{array}\right],$
\end{spacing} \NI $ D_{ij}$'s are diagonal matrices with positive entries and
	\begin{spacing}{1.5} \NI
 $U=\left[\begin{array}{rrrrr} 
	U_{11} & 0 & 0 & \cdots & 0\\
	U_{21} & U_{22} & 0  & \cdots & 0\\	
	U_{31} & U_{32} & U_{33} & \cdots & 0\\
	\vdots & \vdots & \vdots & \vdots & \vdots\\
	U_{n1} & U_{n2} & U_{n3} & \cdots & U_{nn}\\
	\end{array}\right],$
	\end{spacing}
	\NI $U_{ij}$'s are matrices with nonnegative entries.\\
	 $Y$ can be written as $D-V.$ Then the matrix product $NX$ can be written as $D-V,$ where 
	 	\begin{spacing}{1.5} \NI
	  $V=\left[\begin{array}{rrrrr} 
	V_{11} & 0 & 0 & \cdots & 0\\
	V_{21} & V_{22} & 0  & \cdots & 0\\	
	V_{31} & V_{32} & V_{33} & \cdots & 0\\
	\vdots & \vdots & \vdots & \vdots & \vdots\\
	V_{n1} & V_{n2} & V_{n3} & \cdots & V_{nn}\\
	\end{array}\right],$
\end{spacing}
\NI	$ V_{ij}$'s are matrices with nonnegative entries. \vsp \NI As $X,Y$ are block triangular  $K$-matrices, so they are  $P$-matrices. So $v(X)>0, v(Y)>0.$ Let $r \geq 0$ is the value of $X^T$ and  $s \geq 0$ is the value of $Y^T.$ Then $0<r^TX+s^TY=(r^T+s^TN)X=p^TX=p^T(D-U)\\=p^T(D-U)+y^T(-ND+NU+D-V) ,\ \ $ since $N(D-U)=D-V$\\
	$=(p^T-y^TN)(D-U)+y^T(D-V) $\\ $\leq (y^T(I-N)+p^T)D,\ \ $ since $-y^TN+p^T \geq 0, U \geq 0, V \geq 0, y \geq 0.$ \\ Now $D_{ij}$'s are diagonal matrices with positive entries and $D$ is formed with the block matrices $D_{ij}$'s. Hence $y^T(I-N)+p^T>0.$ By lemma \ref{lemm}, $x$ solves the LCP$(N,q),$ which is a solution of $\min $ \ $p^Tx $ subject to $Nx+q \geq 0,$ $x \geq 0.$ 
	
\end{proof}

\begin{lemma}\label{ttth}
	Let $N$ be a hidden block triangular $K$-matrix. Consider the LCP$(\cal{N},$$\bar{q}),$ where $\cal{N}=\left[\begin{array}{rr} 
    0 & -N^T \\
	N & 0 \\	
	\end{array}\right],$ \ $ \bar{q}=\left[\begin{array}{rr} 
r+N^Ts\\
	q\\	
	\end{array}\right] $ and $r,s$ as mentioned in  theorem \ref{tthh}. If $\left[\begin{array}{rr} 
	x\\
	y\\	
	\end{array}\right] \in \text{FEA}(\cal{N},$$ \bar{q}),$ then $(I-N^T)y+p >0,$ where $p=r+N^Ts.$
\end{lemma}
\begin{proof}
	Suppose $\left[\begin{array}{rr} 
	x\\
	y\\	
	\end{array}\right] \in \text{FEA}(\cal{N},$$ \bar{q}).$ Since $N$ is a hidden block triangular $K$-matrix, there exist two block triangular $K$-matrices $X$ and $Y$ such that $NX=Y$ and $r,s \geq 0,$  $r^TX+s^TY>0.$ Let $X=D-U$ and $Y=D-V,$ where $U$ and $V$ are two square matrices with all nonnegative entries and $D$ is a block triangular diagonal matrix with positive entries as mentioned in theorem \ref{tthh}. Then $0<r^TX+s^TY=r^TX+s^TNX=p^T(D-U)=p^T(D-U)+y^T(Y-NX)=p^T(D-U)+y^T(D-V-N(D-U))=(-y^TN+p^T)(D-U)+y^T(D-V)\leq (y^T(I-N)+p^T)D \ $ since $\left[\begin{array}{rr} 
	x\\
	y\\	
	\end{array}\right] \in \text{FEA}(\cal{N},$$ \bar{q}),$ \ $U\geq 0, V\geq 0$. Since $D$ is a positive block triangular diagonal matrix, $ (I-N^T)y+p >0.$
\end{proof}
\begin{theorem}\label{ttthhh}
	LCP$(\cal{N},$$\bar{q})$ has a solution iff LCP$(N,q)$ has a solution.
\end{theorem}
\begin{proof}
	Suppose LCP$(\cal{N},$$\bar{q})$ has a solution. Let $\bar{z}=$$\left[\begin{array}{rr} 
		x\\
		y\\	
	\end{array}\right] \in \text{SOL}(\cal{N},$$ \bar{q}).$ From the complementarity condition it follows that $x^T(p-N^Ty)+y^T(Nx+q)=0.$ Since $p-N^Ty , Nx+q , x, y \geq 0,$ and $ x^T(p-N^Ty)=0, y^T(Nx+q)=0.$ By lemma \ref{ttth}, it follows that $y+(p-N^Ty)>0.$ This implies for all $i$ either $(p-N^Ty)_i>0$ or $y_i>0.$ Now if $(p-N^Ty)_i>0,$ then $x_i=0.$ If $y_i>0$ then $(q+Nx)_i=0.$ This implies $x_i(q+Nx)_i=0 \ \forall \ i.$ Therefore $x$ solves LCP$(N,q)$.\\
	Conversely, $x$ solves LCP$(N,q).$ Let $y=s,$ where $s$ as mentioned in theorem \ref{tthh}. Here $ p-N^Ty=r+N^Ts-N^Ty=r+N^Ts-N^Ts=r \geq 0.$ So $\bar{z}=$$\left[\begin{array}{rr} 
	x\\
	s\\	
	\end{array}\right] \in \text{FEA}(\cal{N},$$ \bar{q}).$ Further $\cal{N}$ is $PSD$-matrix, which implies that $\cal{N}$$\in Q_0.$ Therefore  $\bar{z}$ solves the	LCP$(\cal{N},$$\bar{q}).$	
\end{proof}
\begin{theorem}
	All hidden block triangular $K$-matrices are $Q_0.$
\end{theorem}
\begin{proof}
	Let $N$ ba a hidden block triangular $K$-matrix. It is clear that feasibility of LCP$(N,q)$ implies the feasibility of LCP$(\cal{N},$$\bar{q}).$ Note that  $\cal{N}$$\in Q_0.$ This implies that the feasible point of LCP$(\cal{N},$$\bar{q})$ is also a solution of LCP$(\cal{N},$$\bar{q}).$ Hence by theorem \ref{ttthhh} feasibility of LCP$(N,q)$ ensures the solvability of LCP$(N,q).$ Therefore $N$ is a $Q_0$-matrix. 
\end{proof}
\begin{remk}
	\begin{spacing}{1.5}
	Let $M=\left[\begin{array}{rrrrr} 
	M_{11} & 0 & 0 & \cdots & 0\\
	M_{21} & M_{22} & 0  & \cdots & 0\\	
	M_{31} & M_{32} & M_{33} & \cdots & 0\\
	\vdots & \vdots & \vdots & \vdots & \vdots\\
	M_{n1} & M_{n2} & M_{n3} & \cdots & M_{nn}\\
	\end{array}\right] ,$ where $M_{ij} \in R^{m\times m}$ are $K$-matrices. 
\end{spacing}
\begin{spacing}{1.5}
	Let $z=\left[\begin{array}{rrrrr}
	z_1\\
	z_2\\
	z_3\\
	\vdots\\
	z_n\\
	\end{array}\right]$ and $q=\left[\begin{array}{rrrrr}
	q_1\\
	q_2\\
	q_3\\
	\vdots\\
	q_n\\
	\end{array}\right],$ where $z_i, q_i\in R^m .$ \end{spacing}
\begin{spacing}{1.5}
	Then $Mz+q=\left[\begin{array}{rrrrr}
	M_{11}z_1+q_1\\
	M_{21}z_1+M_{22}z_2+q_2\\
	M_{31}z_1+M_{32}z_2+M_{33}z_3+q_3\\
	\vdots\\
	M_{n1}z_1+M_{n2}z_2+M_{n3}z_3+ \cdots +M_{nn}z_n+q_n\\
	\end{array}\right].$ 
\end{spacing}
	First we solve LCP$(M_{11},q_1)$ and get the solution $w_1=M_{11}z_1+q_1, {w_1}^Tz_1=0.$ Then we solve LCP$(M_{22},M_{21}z_1+q_2)$ and get the solution $w_2=M_{22}z_2+M_{21}z_1+q_2, {w_2}^Tz_2=0.$ Finally we solve LCP$(M_{nn},M_{n1}z_1+M_{n2}z_2+M_{n3}z_3+\cdots+M_{n(n-1)}z_{n-1}+q_n)$ and get the solution $w_n=M_{nn}z_n+M_{n1}z_1+M_{n2}z_2+M_{n3}z_3+\cdots+M_{n(n-1)}z_{n-1}+q_n, {w_n}^Tz_n=0.$ \begin{spacing}{1.5} So $w=\left[\begin{array}{rrrrr}
	w_1\\
	w_2\\
	w_3\\
	\vdots\\
	w_n\\
	\end{array}\right]$ and $z=\left[\begin{array}{rrrrr}
	z_1\\
	z_2\\
	z_3\\
	\vdots\\
	z_n\\
	\end{array}\right]$ solve LCP$(M,q).$ \end{spacing}

\end{remk}
\section{Conclusion}
In this article, we introduce the class of block triangular $K$-matrix and the class of hidden block triangular $K$-matrix in the context of solution of linear complementarity problem. We call these two classes jointly as $K$-type block matricces. We show that the linear complementarity problem with $K$-type block matrix is solvable by linear program. The linear complementarity problem with block triangular $K$-matrix is also processable by Lemke's algorithm as well as criss-cross method. We show that the hidden block triangular $K$-matrix is a $Q_0$-matrix.

\section*{Acknowledgement}
The author A. Dutta is thankful to the Department of Science and Technology, Govt. of India, INSPIRE Fellowship Scheme for financial support.
\vsp

\bibliographystyle{plain}
\bibliography{bibfile}

\end{document}